\documentclass{amsart}[12pt]

\usepackage{amsmath}
\usepackage{amsfonts}
\usepackage{amssymb}

\newtheorem{theorem}{Theorem}
\newtheorem{lemma}{Lemma}

\theoremstyle{definition}

\begin{document}
\title[ Norm-attaining integral operators]%
{Norm-attaining integral operators on analytic function spaces}

\author{ chengji xiong \& junming liu}

\address{School of Mathematical Science, Huaiyin normal university, Huaian, Jiangsu, 223300, P.~R.~China} \email{chengji\_xiong@yahoo.com}
\address{Department of Mathematics, Shantou University, Shantou, Guangdong,
515063, \newline P.~R.~China} \email{08jmliu@stu.edu.cn}

\begin{abstract}
Any bounded analytic function $g$ induces a bounded integral operator $S_g$ on the Bloch space, the Dirichlet space and $BMOA$ respectively.
$S_g$ attains its norm on the Bloch space and $BMOA$ for any $g$, but does not attain its norm on the Dirichlet space for non-constant $g$.
Some results are also obtained for $S_g$ on the little Bloch space, and for another integral operator $T_g$ from the Dirichlet space to the Bergman space.
\end{abstract}
\thanks{This work was supported by NNSF of China (Grant No. 11171203).}

\maketitle

\section{Introduction}
Given two Banach spaces $X$ and $Y$, a bounded linear operator $T$ from $X$ to
$Y$ is said to be norm-attaining if there is an $x\in X$ with $\|x\|=1$
such that $\|Tx\|=\|T\|$. Such an $x$ is called an extremal point (or extremal
function for the function space $X$) for the norm of $T$. Norm-attaining
operators have been studied extensively by many authors, such as \cite{SW},
\cite{KOS} and \cite{SS}.
Especially, there are several papers about norm-attaining operators on
analytic function spaces. For examples, norm-attaining composition
operators on Hardy spaces, Bergman spaces and the Dirichlet space were discussed in \cite{HC} and \cite{HC1} respectively.
Recently, Mar\'{i}a considered norm-attaining composition operators on the Blcoh space and the little Bloch space in \cite{MM}.

Let $\mathbb{D}$ be the open unit disc in the complex plane $\mathbb{C}$
and $H(\mathbb{D})$ the space of all analytic functions on $\mathbb{D}$.
Any $g\in H(\mathbb D)$ induces two integral operators $S_{g}$ and $T_g$ on $H(\mathbb D)$ as following
$$S_{g}f(z)\triangleq\int_{0}^{z}f'(w)g(w)dw\quad\mbox{and}\quad T_{g}f(z)\triangleq\int_{0}^{z}f(w)g'(w)dw.$$
Let $M_g$ denote the multiplication operator $M_g(f)=fg$. Then
$$(S_g+T_g)f=fg-f(0)g(0)=M_g(f)-f(0)g(0).$$
If $g$ is constant, then all results about $S_{g}$, $T_g$ or $M_g$  are trivial.
In general, $g$ is assumed to be non-constant.

Both integral operators have been studied by several authors. Most recently, bounded below of
these integral operators are studied on some analytic functions spaces. See \cite{AA} and its
references for more materials.

In this paper, the norm-attaining integral operator $S_g$ is investigated on several classic analytic function spaces.
In the section 2, the operator norm of $S_g$ on the Bloch space is estimated. We show that $S_g$ is norm-attaining for any $g$ on the Bloch space, but does not attain its norm on the little Bloch space for non-constant $g$. In the section 3, the operator norm of $S_g$ on the Dirchlet space is obtained. We show that $S_g$ does not attain its norm on the Dirichlet space for non-constant $g$. Another integral operator $T_g$ from the Dirichlet space to the Bergman space, is norm-attaining if and only if $g$ is linear. In the last section, we estimate the operator norm of $S_g$ on $BMOA$ and prove that $S_g$ is norm-attaining on $BMOA$ for any $g$.

\section{\bf Norm-attaining integral operators on $\mathcal B_0$ and $\mathcal B$}
The Bloch space $\mathcal B$ is the set of all functions $f\in H(\mathbb D)$ with
$$||f||_{\mathcal B}=|f(0)|+\sup_{z\in\mathbb D}(1-|z|^2)|f'(z)|<\infty.$$
The set of all functions in $\mathcal B$ satisfying
$$\lim_{|z|\rightarrow 1}(1-|z|^2)|f'(z)|=0$$
is called as the little Bloch space $\mathcal B_0$. The Bloch space $\mathcal B$ becomes a Banach
space under the norm $||\cdot||_{\mathcal B}$ and the little Bloch space
is a closed subspace of $\mathcal B$. See \cite{zhu} for more information on Bloch spaces.

A sequence $\{z_{n}\}\subset \mathbb{D}$ is said to be thin if
$$\lim_{k\rightarrow \infty}\prod_{j, j\neq k}\Big|\frac{z_{j}-z_{k}}{1-\overline{z_{j}}z_{k}}\Big|=1\ .$$
It is trivial that a thin sequence implies $z_n\rightarrow\partial\mathbb{D}$.
The Blaschke product
$$B(z)=\prod_{n=1}^{\infty}\frac{|z_{n}|}{z_{n}}\frac{z_{n}-z}{1-\overline{z}_{n}z}$$
is called thin if its zeros $\{z_{n}\}$ form a thin sequence. Gorkin and Mortini \cite{GM} showed that any sequence in $\mathbb{D}$ tending to the boundary admits a thin subsequence as following.

\noindent {\bf Lemma A}. If $\{z_{n}\}$ is a sequence in $\mathbb{D}$ and
$z_{n}\rightarrow \partial \mathbb{D}$, then there exists a thin
Blaschke product $B$ whose zeros are contained in the set $\{z_{n}:
n\in\mathbb{ N}\}$.

Let $B(z)$ be the thin Blaschke product with zeros $\{z_{n}\}$. Direct computation gives
$$(1-|z_{n}|^{2})|B'(z_{n})|=\prod_{k, k\neq n}\Big|\frac{z_{n}-z_{k}}{1-\overline{z}_{k}z_{n}}\Big|\ .$$
Using the definition of thin sequence and Schwarz Pick Lemma, it is easy to see
$$1\ge \sup_{z\in\mathbb{D}}(1-|z|^{2})|B'(z)|\ge \lim_{n\rightarrow \infty}(1-|z_{n}|^{2})|B'(z_{n})| =1\ ,$$
so the semi-norm of $B$ equals $1$.

In \cite{YR},  Yoneda gave the following result.

\noindent {\bf Lemma B}. A integral operator $S_{g}$ is bounded on the Bloch
space $\mathcal B$ $(\mathcal B_0)$ if and only if the inducing
function $g$ belongs to $H^{\infty}$, the set of all bounded
analytic functions on $\mathbb D$.

To determine whether or not an operator is norm-attaining, we need
to calculate the operator norm of it. For the operator norm of $S_g$
on $\mathcal B$ $(\mathcal B_0)$, we obtained following result.

\begin{lemma} If a integral operator $S_{g}$ is bounded on
$\mathcal B$ $(\mathcal B_0)$, then
$$\|S_{g}\|=\sup_{z\in \mathbb{D}}|g(z)|\ .$$
\end{lemma}

\begin{proof}By the Lemma B, we know that $g$ is bounded.
For any $f\in\mathcal B$ with $\|f\|_{\mathcal B}=1$,
$$\|S_{g}f\|_{\mathcal{B}}=\sup_{z\in \mathbb D}(1-|z|^{2})|f'(z)||g(z)|
\le\|f\|_{\mathcal B}\cdot\sup_{z\in \mathbb{D}}|g(z)|=\sup_{z\in \mathbb{D}}|g(z)|\ ,$$
which implies
$$\|S_{g}\|\leq \sup_{z\in \mathbb{D}}|g(z)|\ .$$
Now we need only to show the reverse inequality. Denote $a=\sup_{z\in \mathbb{D}}|g(z)|$. Given any $\epsilon>0$, there exists $z_{0}\in \mathbb{D}$ such
that $|g(z_{0})|\geq a-\epsilon $.
Let
 $$f_{z_{0}}(z)=\frac{z_{0}-z}{1-\overline{z}_{0}z}-z_0.$$
Then $f_{z_0}\in\mathcal{B}_{0},\|f_{z_{0}}\|_{\mathcal{B}}=1$ and
$(1-|z_0|^2)|f'_{z_0}(z_0)|=1$. We have
\begin{equation}
\begin{split}\nonumber
\|S_{g}\| &\geq \|S_{g}f_{z_{0}}(z)\|_{\mathcal{B}}
\\&\geq (1-|z_{0}|^{2})|f'_{z_{0}}(z_{0})|\cdot|g(z_{0})|
 \\&\geq a-\epsilon\ .
 \end{split}
 \end{equation}
Since $\epsilon$ is arbitrary, the reverse inequality holds.
\end{proof}

\begin{theorem}If $g\in H^{\infty}$ is not constant, then the integral operator $S_{g}$ does not attain its norm on $\mathcal{B}_0$.
\end{theorem}

\begin{proof}Suppose that $S_{g}$ is norm-attaining on $\mathcal{B}_{0}$,
then there exists an extremal function $f\in\mathcal{B}_{0}$ with
$\|f\|_{\mathcal B}=1$ satisfying
$\|S_g\|=\|S_gf\|_{\mathcal{B}}$. By Lemma 1, we have
$$\sup_{z\in \mathbb{D}}|g(z)|=\|S_{g}\|=\|S_{g}f\|_{\mathcal{B}}
=\sup_{z\in \mathbb{D}}(1-|z|^{2})|f'(z)|\cdot|g(z)|\le\sup_{z\in \mathbb{D}}|g(z)|.$$
By definition, $f$ belongs to $\mathcal{B}_0$ implies that
$$\lim_{|z|\rightarrow 1^{-}}(1-|z|^{2})|f'(z)|=0\ .$$
Hence $g$ must attains its maximal modulus at some point $z_{0}\in \mathbb{D}$, which contradicts with Maximal Modulus Theorem since $g$ is not constant by assumption.
\end{proof}

\begin{theorem}\label{a} For any $g\in H^{\infty}$, the integral operator $S_{g}$ attains its norm on Bloch space $\mathcal{B}\,$. Furthermore, a function $f\in \mathcal{B}$ with $\|f\|_{\mathcal{B}}=1$ is extremal for the norm of $S_g$ if and only if there exists a sequence $\{z_{n}\}\subset \mathbb{D}$ with $z_{n}\rightarrow \partial \mathbb{D}$ such that both conditions
$$\lim_{n\rightarrow \infty}|g(z_{n})|=\sup_{z\in \mathbb{D}}|g(z)|\ \ \mbox{and}\ \ \lim_{n\rightarrow\infty}(1-|z_{n}|^{2})|f'(z_{n})|=1$$
are satisfied.
\end{theorem}

\begin{proof}If $g$ is constant, it is trivial that $S_g$ is norm-attaining. Given any non-constant $g\in H^{\infty}$, we will construct an extremal function for the norm of $S_{g}$ on $\mathcal{B}$. Take a sequence $\{z_{n}\}$ in $\mathbb{D}$
with $\lim_{n\rightarrow \infty}|g(z_{n})|=\sup_{z\in\mathbb{D}}|g(z)|$.
Then $\lim_{n\rightarrow \infty}|z_n|=1$ by Maximal Modulus Theorem.
Using Lemma A, we obtain a thin Blaschke product $B(z)$ whose zeros
is a subsequence of  $\{z_{n}\}$, which is denoted also by $\{z_n\}$ for convenience.
Let $h(z)=B(z)-B(0)$, then $\|h\|_{\mathcal{B}}=1$.
Keep in mind that
$$\lim_{n\rightarrow\infty}(1-|z_{n}|^{2})|B'(z_{n})|=1 \mbox{ and } \lim_{n\rightarrow \infty}|g(z_{n})|=\sup_{z\in\mathbb{D}}|g(z)|\ .$$
We have
\begin{equation}
\begin{split}\nonumber
\|S_{g}\|\ge\|S_{g}h\|_{\mathcal{B}} &= \sup_{z\in
\mathbb{D}}(1-|z|^{2})|B'(z)|\cdot|g(z)|
\\&\geq \lim_{n\rightarrow\infty}(1-|z_{n}|^{2})|B'(z_{n})|\cdot|g(z_{n})|
\\&= \sup_{z\in \mathbb{D}}|g(z)|=\|S_{g}\| \ ,
\end{split}
\end{equation}
where the Lemma 1 is used. So the function $h$ is an extremal
function for the norm of $S_{g}\,$.

Now, if $f$ is an extremal function for the norm of $S_{g}$, then
\begin{equation}
\begin{split}\nonumber
\|S_{g}\| &= \|S_{g}f\|_{\mathcal{B}}
\\&= \sup_{z\in \mathbb{D}}(1-|z|^{2})|f'(z)|\cdot|g(z)|
 \\&\leq \sup_{z\in \mathbb{D}}|g(z)|=\|S_{g}\|.
 \end{split}
 \end{equation}
Since $f\in\mathcal B$ and $\|f\|_{\mathcal B}=1$, there
exists a sequence $\{z_{n}\}$ with $z_{n}\rightarrow \partial
\mathbb{D}$ such that
$$\lim_{n\rightarrow \infty}|g(z_{n})|=\sup_{z\in \mathbb{D}}|g(z)|\ \ \mbox{and}\ \ \lim_{n\rightarrow\infty}(1-|z_{n}|^{2})|f'(z_{n})|=1.\eqno{*}$$

Conversely, suppose that there is a sequence $\{z_{n}\}$ and a unit norm function $f$
such that (*) holds, then
\begin{equation}
\begin{split}\nonumber
\|S_{g}\| \ge\|S_{g}f\|_{\mathcal{B}} &= \sup_{z\in
\mathbb{D}}(1-|z|^{2})|f'(z)|\cdot|g(z)|
\\&\geq \lim_{n\rightarrow\infty}(1-|z_{n}|^{2})|f'(z_{n})|\cdot|g(z_{n})|
 \\&= \sup_{z\in \mathbb{D}}|g(z)|=\|S_{g}\| \ .
 \end{split}
 \end{equation}
So $f$ is an extremal function for the norm of $S_g$.
\end{proof}

\section{\bf Norm attaining integral operators on the Dirichlet space}
In this section, we discuss norm attaining integral operators on the
Dirichlet space. Some related results about another integral operator from the Bergman space to the Dirichlet space are also
obtained. At first, let us review definitions of the Dirichlet space and the Bergman
space.

Denote $dA(z)=\frac{1}{\pi}dxdy$. The Dirichlet space $\mathcal{D}$ consisting of analytic functions
$f\in H(\mathbb D)$ with
$$\int_{\mathbb{D}}|f'(z)|^{2}dA(z)<\infty .$$
The norm of $f\in\mathcal D$ is defined as
$$\|f\|_{\mathcal{D}}=\Big(|f(0)|^{2}+\int_{\mathbb{D}}|f'(z)|^{2}dA(z)\Big)^{1/2}\ .$$
The Bergman space $\mathcal A^{2}$ consisting of analytic functions $f\in H(\mathbb D)$ with
$$ \int_{\mathbb{D}}|f(z)|^{2}dA(z)<\infty .$$
The norm of $f\in\mathcal A^{2}$ is defined as
$$\|f\|_{\mathcal A^{2}}=\Big(\int_{\mathbb{D}}|f(z)|^{2}dA(z)\Big)^{1/2}\ .$$
$\mathcal D$ and $\mathcal A^2$ are Hilbert spaces with
$$\langle f, g\rangle =f(0)\overline{g(0)}+\int_{\mathbb{D}}f'(z)\overline{g'(z)}dA(z)$$
and
$$\langle f, g\rangle =\int_{\mathbb{D}}f(z)\overline{g(z)}dA(z)$$
respectively.

\begin{lemma}
Let $f$ be an analytic function in $\mathbb{D}$ and $0<r<1$. Then
$$|f(0)|^{2}\leq\frac{1}{2\pi}\int_{0}^{2\pi}|f(re^{i\theta})|^{2}d\theta\mbox{ and } |f(0)|^{2}\leq \int_{\mathbb{D}}|f(z)|^{2}dA(z)\ .$$
\end{lemma}
\begin{proof}Let $f(z)=\sum_{k=0}^{\infty}a_kz^k$ be the Taylor series of $f$, then
$$|f(z)|^2=|f(re^{i\theta})|^2=f(re^{i\theta})\overline{f(re^{i\theta})}=\sum_{j=0,k=0}^{\infty}a_j\overline{a_k}r^{j+k}e^{i(j-k)\theta}\ .$$
Direct computation by Taylor coefficients gives the result.
\end{proof}

The more general form of above lemma can be found in Zhu's book
\cite{zhu}. In \cite{AA1}, Austin characterize the boundness of
integral operators $S_{g}$ on the Dirichlet space, and obtaied
following result.

\noindent {\bf Lemma C}. An integral operator $S_g$ is bounded on the Dirichlet space $\mathcal{D}$
if and only if $g$ belongs to $H^{\infty}$.

Now we can give the norm of $S_{g}$ on $\mathcal{D}$.
\begin{lemma}Let $S_{g}$ be a bounded integral operator on $\mathcal{D}$, then
$$\|S_{g}\|=\sup_{z\in \mathbb{D}}|g(z)|\ .$$
\end{lemma}
\begin{proof}Since $S_{g}$ is bounded on $\mathcal{D}$, we can denote $b=\sup_{z\in \mathbb{D}}|g(z)|$ by the Lemma C. For any
$\epsilon>0$, there exists a point $a\in D$ such that
$|g(a)|>b-\epsilon$. Consider the function
$$f_a(z)=\sigma_{a}(z)-a=\frac{a-z}{1-\overline{a}z}-a\ ,$$
then $\|f_a\|_{\mathcal{D}}=1$. By the Lemma 2, we have
\begin{equation}
\begin{split}\nonumber
\|S_{g}\| &\geq \|S_{g}f_a\|_{\mathcal{D}}
\\&=\Big(\int_{\mathbb{D}}|g(z)|^{2}|\sigma_{a}'(z)|^{2}dA(z)\Big)^{1/2}
\\&=\Big(\int_{\mathbb{D}}|g(\sigma_{a}(z))|^{2}dA(z)\Big)^{1/2}
\\&\geq |g(\sigma_a(0))|=|g(a)|>b-\epsilon\ .
\end{split}
\end{equation}
Hence $\|S_g\|\ge\sup_{z\in\mathbb D}|g(z)|$.

The reverse inequality can be obtained by
$$\|S_{g}f\|_{\mathcal D}=\Big(\int_{\mathbb{D}}|f'(z)g(z)|^{2}dA(z)\Big)^{1/2}\leq \|f\|_{\mathcal{D}}\cdot\sup_{z\in\mathbb{D}}|g(z)|$$
for any $f\in\mathcal D$.
\end{proof}
Using above Lemma 3, we obtain main result of this section.
\begin{theorem}
Let $g$ be a non-constant function in $H^{\infty}$. Then the integral operator $S_{g}$ can not attain its norm on the Dirichlet space $\mathcal{D}$.
\end{theorem}

\begin{proof}
Assume that $S_g$ attains its norm and $f$ is an extremal function. Then it is
necessary that $f(0)=0$. Otherwise, we have
\begin{equation}
\begin{split}\nonumber
\|S_g\|&=\|S_{g}f\|_{\mathcal{D}} \\
&=\Big( \int_{\mathbb{D}} |f'(z)|^{2}|g(z)|^{2}dA(z)\Big)^{1/2}
\\&\leq\Big(\int_{\mathbb{D}}|f'(z)|^{2}dA(z)\Big)^{1/2}\cdot \sup_{z\in\mathbb{D}}|g(z)|
\\&<\Big(|f(0)|^{2}+\int_{\mathbb{D}}|f'(z)|^{2}dA(z)\Big)^{1/2}\cdot \sup_{z\in\mathbb{D}}|g(z)|
\\&= \sup_{z\in\mathbb{D}}|g(z)|=\|S_g\|\ ,
\end{split}
\end{equation}
where $\|f\|_{\mathcal D}=1$ and the Lemma 3 is used. Then
$$\int_{\mathbb D}|f'(z)|^{2}dA(z)=1$$
and
$$\|S_g\|^2=\|S_{g}f\|^{2}_{\mathcal{D}}=\int_{\mathbb{D}}|f'(z)|^{2}|g(z)|^{2}dA(z)\ .$$
But this is impossible since
$$\int_{\mathbb{D}}|f'(z)|^{2}\big(|g(z)|^{2}-\|S_g\|^2) dA(z)<0$$
for any non-constant analytic function $g$.
\end{proof}

Consider another integral operator
$T_{g}f(z)=\int_{0}^{z}g'(w)f(w)dw$. Following results can be
deduced similarly with the Lemma 3 and the Theorem 3. We omit proofs
of them.

\begin{lemma}If the integral operator $T_{g}$ is bounded from
$A^{2}$ to $\mathcal{D}$, then
$$\|T_{g}\|=\sup_{z\in \mathbb{D}}|g'(z)|\ .$$
\end{lemma}

This lemma can be proved by replacing $f_a$ with
$$F_a(z)=\frac{1-|a|^{2}}{(1-\overline{a}z)^{2}},\ a\in\mathbb D$$
in the proof of the Lemma 3.

\begin{theorem}
Let $T_{g}$ be a bounded operator from $A^{2}$ to $\mathcal{D}$.
Then $T_{g}$ is norm attaining if and only if $g(z)=az+b$, where $a,
b\in \mathbb{C}$.
\end{theorem}

\section{\bf Norm attaining integral operators on $BMOA$}
The analytic Hardy space $H^{2}$ on the unit disc $\mathbb{D}$ consists of all analytic functions $f\in H(\mathbb D)$ satisfying
$$\|f\|_{H^{2}}=\Big(\sup_{0<r<1}\frac{1}{2\pi}\int_{0}^{2\pi}|f(re^{i\theta})|^{2}d\theta\Big)^{1/2}<\infty\ .$$
For $a\in \mathbb{D}$, denote
$$\sigma_{a}(z)=\frac{a-z}{1-\overline{a}z},\ \ z\in \mathbb{D} \ .$$
Let $BMOA$ denote the space of all analytic functions $f\in
H^{2}$ whose boundary function have bounded mean oscillation. The norm of $f\in BMOA$ can be defined as
$$\|f\|_{BMOA}=|f(0)|+\sup_{a\in \mathbb{D}}\|f\circ\sigma_{a}-f(a)\|_{H^{2}}.$$
There are some other equivalent norms for $BMOA$, see \cite{GD}, for example.
Next is the well-known Littlewood-Paley Identity, whose proof can be found in many books.

\noindent\textbf{Lemma D.} Let $f$ be analytic on unit disc $\mathbb{D}$. Then
$$\|f\|_{H^{2}}^{2}=|f(0)|^2+2\int_{\mathbb{D}}|f'(z)|^{2}\log\frac{1}{|z|}dA(z)\ .$$
By Littlewood-Paley Identity, the norm of $BMOA$ can be expressed as
$$\|f\|_{BMOA}=|f(0)|+\sup_{a\in\mathbb D}\Big\{ \int_{\mathbb{D}}2|f'(z)|^{2}\log\frac{1}{|\sigma_{a}(z)|}dA(z) \Big\}^{1/2}\ .$$

The following lemma is our main tool to compute the norm of integral operators.

\begin{lemma}If $f$ is analytic in $\mathbb{D}$, then
$$|f(0)|^{2}\leq 2\int_{\mathbb{D}}|f(z)|^{2}\log\frac{1}{|z|}dA(z)\ .$$
\end{lemma}
\begin{proof}The proof is similar to that of the Lemma 2.
\end{proof}

In \cite{XJ}, Xiao showed that $S_{g}$ is bounded on $BMOA$ if and
only if $g\in H^{\infty}$. Now we give the norm of $S_g$ on $BMOA$.
\begin{lemma}If $S_g$ is a bounded integral operator on $BMOA$, then
$$\|S_{g}\|=\sup_{z\in\mathbb{D}}|g(z)|\ .$$
\end{lemma}
\begin{proof}Since $S_g$ is bounded on $BMOA$, $g$ is bounded by Xiao \cite{XJ}.
Let $\lambda=\sup_{z\in \mathbb{D}}|g(z)|$. For any $\epsilon>0$, there exists a
point $b\in\mathbb{D}$ such that $|g(b)|>\lambda-\epsilon$. Denote $f_b(z)=\sigma_{b}(z)-b$. Then
\begin{equation}
\begin{split}\nonumber
\|f_b\|_{BMOA}
= & \sup_{a\in\mathbb D}\Big\{ \int_{\mathbb{D}}2|\sigma_b'(z)|^{2}\log\frac{1}{|\sigma_{a}(z)|}dA(z)\Big\}^{1/2}\\
= & \sup_{a\in\mathbb D}\Big\{\int_{\mathbb{D}}2|[\sigma_b(\sigma_a(z))]'|^{2}\log\frac{1}{|z|}dA(z)\Big\}^{1/2}\\
= & \sup_{c\in\mathbb D}\Big\{\int_{\mathbb{D}}2|\sigma_c'(z)|^{2}\log\frac{1}{|z|}dA(z)\Big\}^{1/2}\\
= & \sup_{c\in\mathbb D}\Big\{(\|\sigma_c\|^2_{H^2}-|c|^2)\Big\}^{1/2}\\
= & \sup_{c\in\mathbb D}(1-|c|^2)^{1/2}=1
\end{split}
\end{equation}
by the Lemma D. Hence
\begin{equation}
\begin{split}\nonumber
\|S_{g}\|
& \geq \|S_{g}f_b\|_{BMOA}\\
& =\sup_{a\in\mathbb D}\Big(\int_{\mathbb{D}}2|f_b'(z)g(z)|^{2}\log\frac{1}{|\sigma_a(z)|}dA(z)\Big)^{1/2}\\
& =\sup_{a\in\mathbb D}\Big(\int_{\mathbb{D}}2|g(z)|^{2}|\sigma_b'(z)|^2\log\frac{1}{|\sigma_a(z)|}dA(z)\Big)^{1/2}\\
& =\sup_{a\in\mathbb D}\Big(\int_{\mathbb{D}}2|g(\sigma_b(z))|^2\log\frac{1}{|\sigma_a(\sigma_b(z))|}dA(z)\Big)^{1/2}\\
& \geq\Big(2\int_{\mathbb{D}}|g(\sigma_{b}(z))|^{2}\log \frac{1}{|z|}dA(z)\Big)^{1/2}\\
& \geq |g(b)|\geq \lambda-\epsilon\ ,
\end{split}
\end{equation}
where the Lemma 5 is used.

The reverse inequality can be followed from
$$\|S_{g}f\|_{BMOA}=\sup_{a\in\mathbb D}\Big(\int_{\mathbb{D}}2|f'g|^{2}g(z, a)dA(z)\Big)^{1/2}\leq \|f\|_{BMOA}\cdot
\sup_{z\in\mathbb{D}}|g(z)|\ .$$ The proof is completed.
\end{proof}

Next theorem is the main result of this section.
\begin{theorem}
For any $g\in H^{\infty}$, the integral operator $S_{g}$ is norm-attaining on $BMOA$.
\end{theorem}
\begin{proof}Similar to the proof of the Theorem 2, we need only to construct an extremal function for non-constant $g$. Choose a
sequence $\{z_{n}\}\subset \mathbb{D}$ with $\lim_{n\rightarrow
\infty}|g(z_{n})|=\sup_{z\in \mathbb{D}}|g(z)|$ since $g\in
H^{\infty}$. Using the Lemma A, we obtain a thin Blaschke product
$B(z)$ whose zeros is a subsequence of $\{z_{n}\}$, which is denoted
also by $\{z_{n}\}$ for convenience. Consider $h(z)=B(z)-B(0)$, then
$\|h\|_{BMOA}=1$ by Proposition 2.2 of \cite{LJ}. Note that
$$\lim_{n\rightarrow \infty}(1-|z_{n}|^{2})|h'(z_{n})|=1\ \ \mbox{and}\ \ \lim_{n\rightarrow \infty}|g(z_{n})|
=\sup_{z\in \mathbb{D}}|g(z)|.$$
We have
\begin{equation}
\begin{split}\nonumber
\|S_{g}h\|_{BMOA}&=\sup_{a\in\mathbb D}\Big(\int_{\mathbb{D}}2|h'(z)|^{2}|g(z)|^{2}\log\frac{1}{|\sigma_a(z)|}dA(z)\Big)^{1/2}\\
& =\sup_{a\in\mathbb D}\Big(2\int_{\mathbb{D}}|h'(\sigma_{a}(w))||g(\sigma_{a}(w))|^{2}|\sigma_{a}'(w)|^{2}\log \frac{1}{|w|}dA(w)\Big)^{1/2}\\
& \ge\Big(2\int_{\mathbb{D}}|h'(\sigma_{z_{n}}(w))||g(\sigma_{z_{n}}(w))|^{2}|\sigma_{z_{n}}'(w)|^{2}\log
\frac{1}{|w|}dA(w)\Big)^{1/2}
\\ &\ge (1-|z_{n}|^{2})|h'(z_{n})||g(z_{n})|,
\end{split}
\end{equation}
where the Lemma 5 is used. Let $n\rightarrow\infty$, then
$\|S_gh\|_{BMOA}\ge\sup_{z\in\mathbb D}|g(z)|$ and
$\|S_gh\|_{BMOA}=\sup_{z\in\mathbb D}|g(z)|=\|S_g\|$ since
$\|h\|_{BMOA}=1$. Hence the function $h$ is an extremal function for
the norm of $S_{g}$ and $S_{g}$ is norm-attaining on $BMOA$.
\end{proof}

\end{document}